\documentclass[12pt]{article}

\usepackage[english]{babel}

\usepackage[a4paper,top=2cm,bottom=2.5cm,left=2cm,right=2cm,marginparwidth=2cm]{geometry}

\usepackage{amssymb,amsmath}
\usepackage{amsmath} 
\usepackage{amsthm}
\usepackage{mathrsfs}
\usepackage{amssymb}
\usepackage{siunitx}
\PassOptionsToPackage{hyphens}{url}\usepackage{hyperref}
\usepackage{comment}
\usepackage{cleveref}
\usepackage[utf8]{inputenc}
\usepackage[right]{lineno}
\usepackage{csquotes}
\usepackage{booktabs}
\usepackage{longtable}
\usepackage{adjustbox}
\usepackage{array}
\usepackage{url}
\usepackage{titlesec}
\usepackage{ulem}
\usepackage[numbers]{natbib}
\usepackage{authblk}
\usepackage{xcolor} 

\titleformat{\subsection}
  {\mdseries\itshape\large} 
  {\thesubsection}{1em}{} 

\theoremstyle{definition}
\newtheorem{theorem}{Theorem}[section]
\newtheorem{lemma}[theorem]{Lemma}

\newtheorem{proposition}[theorem]{Proposition}
\newtheorem{example}[theorem]{Example}
\newtheorem{remark}[theorem]{Remark}

\usepackage{enumitem}
\setlist[enumerate]{itemsep=0pt, topsep=0pt}

\usepackage[english]{babel}

\newcommand{\la}{\langle }
\newcommand{\ra}{\rangle}
\newcommand{\F}{\mathbb{F}}
\newcommand{\PG}{\mathrm{PG}}
\newcommand{\cC}{\mathcal{C}}
\newcommand{\rk}{\mathrm{rk}}
\newcommand{\Tr}{\mathrm{Tr}}
\newcommand{\vv}{\mathbf{v}}

\newcommand{\uu}{\mathbf{u}}

\newcommand{\xx}{\mathbf{x}}
\newcommand{\cc}{\mathbf{c}}
\newcommand{\wt}{\mathrm{wt}}

\newcommand{\mS}{\mathcal{S}}

\newcommand{\dimq}{\dim_{\mathbb{F}_q}}
\newcommand{\dimqm}{\dim_{\mathbb{F}_{q^m}}}

\newcommand{\qbinom}[2]{\genfrac{[}{]}{0pt}{}{#1}{#2}_q}
\author[1]{Giovanni Longobardi}
\author[2]{Rocco Trombetti}
\author[3]{Lei Xu}

\affil[1]{Department of Mathematics and Applications "Renato Caccioppoli", University of Naples Federico II,  Naples, Italy, 
giovanni.longobardi@unina.it}
\affil[2]{Department of Mathematics and Applications "Renato Caccioppoli", University of Naples Federico II,  Naples, Italy, 
rocco.trombetti@unina.it}
\affil[3]{School of Mathematics and Statistics,
Beijing Jiaotong University, Beijing,
China,  
xuxinlei@bjtu.edu.cn}

\title{Near-uniform $q$-matroids}

\begin{document}

\date{}
\maketitle

\begin{abstract}
We study the $q$-matroids associated with nondegenerate $\F_{q^m}$-linear near-MRD codes, which we call near-uniform $q$-matroids. Using the correspondence between rank-metric codes and their associated $q$-systems, we determine explicitly their rank functions and characterize their cyclic flats. We show that near-uniform $q$-matroids form a class of representable paving $q$-matroids. and prove an upper bound on the number of their cyclic flats, which is shown to be sharp for certain parameters. Finally, when $n>m$, exploiting the rank distribution of near-MRD codes, we obtain an exact count of the nontrivial cyclic flats of the associated $q$-matroids.
\medskip

\medskip


\end{abstract}

\section{Introduction}
The notion of a $q$-matroid, regarded as the $q$-analogue of a matroid, originates from Crapo’s PhD thesis~\cite{crapo64} and was later reintroduced by  Jurrius and Pellikaan~\cite{jp2018}.
In recent years, $q$-matroids have been studied from different perspectives, including their structural properties, connections with coding theory, and representability and direct sums; see, for instance,~\cite{alde2026,ajnz2024,bc2022, cj2024, gj22,g-lj2025, gj2020,gjhr20,shir2019}.

As in classical matroid theory, $q$-matroids admit several
cryptomorphic descriptions, including axiom systems in terms of rank
functions, independent spaces, bases, circuits, flats, spanning spaces,
closure operators, hyperplanes, and open spaces~\cite{bc2022}.
More recently, Alfarano and Byrne established a cryptomorphic
description in terms of cyclic flats~\cite{alfbyrne}, showing that a
$q$-matroid is uniquely determined by its cyclic flats together with
their ranks.

There is a strong connection between $q$-matroids and rank-metric codes~\cite{gjhr20,jp2018,shir2019}. Nondegenerate rank-metric codes also admit an equivalent geometric description in terms of $q$-systems; see~\cite{randrianarisoa2020,sv2020}. The representability of $q$-matroids was investigated in~\cite{g-lj2025}, where several examples of non-representable $q$-matroids were constructed. Similarly to the classical setting, not all $q$-matroids are representable~\cite{gj22}. It was shown that asymptotically almost all $q$-matroids are non-representable~\cite{deku2026}, further emphasizing the question of which $q$-matroids admit a representation. Beyond representability over extension fields, a notion of multilinear representability for $q$-matroids in terms of matrix rank-metric codes was introduced and investigated in~\cite{alde2026}.

Paving $q$-matroids are $q$-matroids of rank $k$ whose circuits have dimension at least $k$. Several constructions of paving $q$-matroids have been considered in~\cite{gj22}, including examples that are not representable. Among representable paving $q$-matroids, uniform $q$-matroids form a special class whose representability can be characterized in terms of maximum rank distance (MRD) codes~\cite{g-lj2025}.

In analogy with the Hamming metric setting, rank-metric codes that are as close as possible to being MRD were introduced in~\cite{dlc2018}. Their geometric properties were subsequently investigated in~\cite{mnt23}, where these codes were referred to as \textit{near-MRD} codes. Building on these results and the connection between rank-metric codes and $q$-matroids, we study the $q$-matroids associated with near-MRD codes, which we call \textit{near-uniform} $q$-matroids.

The main contributions of this paper are devoted to the structure of these $q$-matroids. We determine their rank function (see Theorem~\ref{th:near-matroid}) and characterize their cyclic flats. 

 In Theorem~\ref{thm:bound-Dk-near-MRD}, we further obtain an upper bound on the number of cyclic flats which is attained for certain parameters. For $n>m$, applying the formula for the wight distribution of near-MRD code (\cite[Corollary 28]{dela18}) gives  the exact number of cyclic flats of associated near-uniform $q$-matroids, see Theorem~\ref{th:sizebound}.

\medskip
\noindent \textbf{Outline of the paper}. The paper is organized as follows. In Section ~\ref{pre}, we provide the necessary preliminaries on $q$-matroids, rank-metric codes, and $q$-systems. In Section~\ref{sec:near-MRD codes and their $q$-matroids}, we recall the main properties of near-MRD codes and study the $q$-matroids associated with them. In particular, we determine their rank functions, characterize their cyclic flats, and derive an upper bound on their number. In Section~\ref{determining}, we focus on $q$-matroids arising from optimal near-MRD codes and, by using their rank distribution \cite[Corollary 28]{dela18}, we compute the exact number of cyclic flats of the relevant $q$-matroids for $n >m$. Finally, Section~\ref{sec:conclusions} concludes the paper and presents some open problems.

\section{Preliminaries}\label{pre}

\subsection{\texorpdfstring{$q$-matroids}{q-matroids} and representability}

Let $n \in \mathbb{N}$ and $E$ be an $\F_q$-vector space of dimension $n$.  A $q$-\textit{matroid} is a pair $\mathcal{M}=(E,\rho)$ where $\rho$ is a function from the lattice $\mathscr{L}(E)$ of the $\F_q$-vector subspaces of $E$ to $\mathbb{N}$ such that, for all $A,B \in \mathscr{L}(E)$, holds that:

\begin{enumerate}
\item[(1)] Boundedness: $0 \leq \rho(A) \leq \operatorname{dim} A$.
\item[(2)] Monotonicity: $A \leq B \Rightarrow \rho(A) \leq \rho(B)$.
\item[(3)] Submodularity: $\rho(A+B)+\rho(A \cap B) \leq \rho(A)+\rho(B)$.
\end{enumerate}

The map $\rho$ is called \textit{rank function} and the value $\rho(\mathcal{M}):=\rho(E)$ is the \textit{rank} of the $q$-matroid. The value $h(\mathcal{M})=\dim_{\F_q}(E)$ is known as the \textit{height} of $\mathcal{M}$. A $1$-dimensional subspace of $E$ is called a \textit{loop} if the rank function takes the value $0$ on it.  
A subspace $A \in \mathscr{L}(E)$ is \textit{independent} if $\rho(A) = \dim_{\mathbb{F}_q}(A)$ and \textit{dependent} otherwise.  
The inclusion–minimal dependent subspaces are called \textit{circuits}.  
A subspace $A \in \mathscr{L}(E)$ is said to be a \textit{flat} if $\rho(A + X) > \rho(A)$ for every $1$-dimensional subspace $X \leq E$ not contained in $A$, and it is called \textit{open} if it is a sum of circuits. A subspace $A \in \mathscr{L}(E)$ is called \textit{cyclic} if $\rho(A)=\rho(H)$ for any hyperplane $H$ of $A$. Finally, a subspace is said to be a \textit{cyclic flat} if it is open and flat, or equivalently, if it is cyclic and flat \cite[Proposition 3.14]{alfbyrne}.

If $\mathcal{M}$ is a $q$-matroid, we will denote the set of its cyclic flats  by $\mathscr{Z}(\mathcal{M})$.

\begin{lemma}[\cite{alfbyrne}, Lemma 3.28]\label{alfbyrne-lemma228}
    Let $\mathcal{M}=(E, \rho)$ be a $q$-matroid, then the following hold.
    
    \begin{enumerate}
       \item[$(i)$] $I \in \mathscr{L}(E)$ is independent if and only if for every cyclic flat $Z, \dimq (I \cap Z) \leq \rho(Z)$.
\item[$(ii)$] $C \in \mathscr{L}(E)$ is a circuit if and only if $C$ is an inclusion-minimal space such that there exists a cyclic flat $Z$ satisfying $C \leq Z$ and $\dimq(C)=\rho(Z)+1$.
    \end{enumerate}
\end{lemma}

\begin{lemma}[\cite{gj2022}, Proposition 3.3]\label{lesskkk}
    Let $\mathcal{M}=(E, \rho)$ be a $q$-matroid. Then for all $V \in \mathscr{L}(E)$
$$
\rho(V)=\dimq V \Longrightarrow \rho(W)=\dimq  W \text { for all } W \leqslant V.$$
\end{lemma}

In the remainder we will always assume that $E=\F_q^n$. Let $\mathcal{M}_1=\left(\mathbb{F}_q^n, \rho_1\right)$ and $\mathcal{M}_2=\left(\mathbb{F}_q^n, \rho_2\right)$ be $q$-matroids. We say that $\mathcal{M}_1$ and $\mathcal{M}_2$ are \textit{equivalent} if there exists an isomorphism $\varphi: \mathbb{F}_q^n \rightarrow \mathbb{F}_q^n$ such that $\rho_1(A)=\rho_2(\varphi(A))$ for all $A \in \mathscr{L}(\mathbb{F}_q^n)$. In this case, we write $\mathcal{M}_1 \cong \mathcal{M}_2$.

Now, let $\mathbf{G} \in \mathbb{F}_{q^m}^{k \times n}$ be a $k \times n$ matrix with entries over $\mathbb{F}_{q^m}$ and for every $U \in \mathscr{L}\left(\mathbb{F}_q^n\right)$ with $\dimq U = u$, let $A^U \in \F_q^{n \times u}$ be a matrix whose columns span $U$. In \cite[Lemma 4.2]{gj22}, it was proven that the map
$$\rho_{\mathbf{G}}: U \in \mathscr{L}\left(\mathbb{F}_q^n\right) \longrightarrow  \mathrm{rk}_{\mathbb{F}_{q^m}}(\mathbf{G} A^U) \in \mathbb{N}$$
is the rank function of a $q$-matroid. In the following, we denote this $q$-matroid by $\mathcal{M}_\mathbf{G}$ and will refer to it as  the $q$-matroid {\it represented} by $\mathbf{G}$. 

A $q$-matroid $\mathcal{M}$ with ground vector space $\mathbb{F}_q^n$ is called $\mathbb{F}_{q^m}$-{\it representable} if $\mathcal{M}\cong\mathcal{M}_\mathbf{G}$ for some full rank matrix $\mathbf{G} \in \mathbb{F}_{q^m}^{\rho(\mathcal{M})\times n}$. In such a case, we say that the $q$-matroid $\mathcal{M}_\mathbf{G}=(\F_q^n,\rho_{\mathbf{G}})$ is an $\F_{q^m}$-\textit{representation} of $\mathcal{M}$.

For an $\mathbb{F}_{q}$-subspace $V \subseteq \mathbb{F}_{q^m}^k$, we define its $\mathbb{F}_{q^m}$-\textit{rank} as the integer
\begin{equation}\label{notation}
\rho(V)=\operatorname{dim}_{\mathbb{F}_{q^m}}\left(\langle V\rangle_{\mathbb{F}_{q^m}}\right),
\end{equation}
that is, the $\mathbb{F}_{q^m}$-dimension of the $\mathbb{F}_{q^m}$-span of any $\mathbb{F}_{q}$-basis of $V$. 

\subsection{$q$-Matroids represented via rank-metric codes}\label{subsec:rm}

In this section, we discuss representability of $q$-matroids via rank-metric codes. Let $\mathbb{F}_q$ be the finite field with $q$ elements and let $m, n$ be positive integers. The vector space $\mathbb{F}_{q^m}^n$ can be endowed with the \textit{rank metric}, defined as 
$$
\mathrm{d}_{\mathrm{rk}}(\textbf{x}, \textbf{y})=\operatorname{dim}_{\mathbb{F}_q}\left\langle x_1-y_1, \ldots, x_n-y_n\right\rangle_{\mathbb{F}_q},
$$
for any $\textbf{x}=\left(x_1, \ldots, x_n\right)$ and $\textbf{y}=\left(y_1, \ldots, y_n\right) \in \mathbb{F}_{q^m}^n$.
An $[n,k]_{q^m / q}$ \textit{rank-metric linear code} $\mathcal{C}$ is a $k$-dimensional subspace of $\mathbb{F}_{q^m}^n$. The \textit{minimum rank distance} of $\mathcal{C}$, denoted by $d(\cC)=\mathrm{d}_{\mathrm{rk}}(\mathcal{C})$, is defined as
$$
\mathrm{d}_{\mathrm{rk}}(\mathcal{C}):=\min \left\{\mathrm{d}_{\mathrm{rk}}(\xx, {\bf y}) \mid \xx, {\bf y} \in \mathcal{C}, \xx \neq {\bf y}\right\}.
$$
If $\cC$ has minimum rank distance $d$, we also say that $\cC$ is an $[n,k,d]_{q^m/q}$ code.

The parameters $n, m, k,$ and $d$ of a  rank-metric code are related by the following inequality known
as Singleton bound (see e.g. \cite{delsarte}):
$$mk \leq \min \{m(n - d + 1), n(m - d + 1)\}.$$ 

Codes whose parameters meet the Singleton bound with equality are called {\it maximum rank distance} codes or {\it MRD} codes for short.

For a vector $\mathbf{v} \in \mathbb{F}_{q^m}^n$ and an ordered basis $\Gamma=\left\{\gamma_1, \ldots, \gamma_m\right\}$ of $\mathbb{F}_{q^m}$ seen as $m$-dimensional vector space over $\F_q$, let $\Gamma(\mathbf{v}) \in \mathbb{F}_q^{n \times m}$ be the matrix defined by
$$
v_i=\sum_{j=1}^m \Gamma(\mathbf{v})_{i j} \gamma_j .
$$

The \textit{support} of $\mathbf{v}$ is the column space of $\Gamma(\mathbf{v})$ for any basis $\Gamma$ and  we denote it by $\operatorname{supp}(\mathbf{v})$. The \textit{rank-weight} of $\mathbf{v} \in \mathbb{F}_{q^m}^n$ is defined as $\operatorname{rk}(\mathbf{v})=\operatorname{dim}_{\mathbb{F}_q}(\operatorname{supp}(\mathbf{v}))$. 
The \textit{rank distribution} of $\cC$ is the collection $(A_i(\cC))_{i\in\mathbb{N}}$, where
$A_i(\cC):= |\{\vv\in\cC: \mathrm{rk}(\vv)=i\}|$.

A \textit{generator matrix} for an $[n, k,d]_{q^m / q}$ code $\mathcal{C}$ is a full-rank matrix $\textbf{G}=(u_1| \cdots |u_n) \in \mathbb{F}_{q^m}^{k \times n}$ such that $\mathcal{C}=\operatorname{rowsp}(\mathbf{G})$. If the columns of one (and hence any) generator matrix of $\mathcal{C}$ are $\mathbb{F}_q$-linearly independent, then $\mathcal{C}$ is said to be \textit{nondegenerate}. Two $[n, k, d]_{q^m / q}$ rank-metric codes $\mathcal{C}$ and $\mathcal{C}^{\prime}$ are {\it equivalent} if and only if there exists $A \in \mathrm{GL}\left(n, \mathbb{F}_{q}\right)$ such that $\mathcal{C}^{\prime}=\mathcal{C} A=\{\mathbf{v} A: \textbf{v} \in \mathcal{C}\}$.

Let $\cC$ be an $[n,k]_{q^m/q}$ rank-metric code. The \textit{dual code}  of $\cC$ is defined as 
\begin{equation}
    \cC^{\perp}=\{\textbf{v} \in \F_{q^m}^n \colon \textbf{v} \cdot \textbf{c}= \textbf{0} \,\,\forall \cc \in \cC \}
\end{equation}
where $\cdot$ is the standard inner product of $\F_{q^m}^n$. The $\cC^\perp$ is an $[n,n-k]_{q^m/q}$ rank-metric code and a full-rank matrix $\textbf{H} \in \F_{q^m}^{(n-k) \times n}$  such that  $\cc \textbf{H}^t=\textbf{0}$ for any $\cc \in \cC$ is called \textit{parity-check matrix} of $\cC$. Note that a generator matrix $\textbf{G}$ of $\cC$ is a parity-check matrix for $\cC^\perp$.

\begin{theorem}[\cite{g85}, Theorem 1] \label{gabidulin} Let $\cC \subset \F_{q^m}^n$ be an $[n,k]_{q^m/q}$ rank-metric code with parity check matrix $\mathbf{H}$. Then $\cC$ has rank distance $d$ if and only if for any $(d-1) \times n$ matrix $Y \in \F_q^{(d-1) \times n}$ of rank $d-1$
$$
\rk_{\mathbb{F}_{q^m}}(Y \mathbf{H}^t)=d-1
$$
and there exists a matrix $Y_0 \in \F_q^{d \times n}$ of rank $d$ such that
$$
\mathrm{rk}_{\mathbb{F}_{q^m}}(Y_0 \mathbf{H}^t)<d.
$$
\end{theorem}

A $[n,k,d]_{q^m/q}$-system $\mathcal S$ of $\F_{q^m}^k$ is an $n$-dimensional $\F_q$-subspace of $\F_{q^m}^k$ such that $\langle \mathcal S\rangle_{\F_{q^m}} = \mathbb{F}^k_{q^m}$ and 
\[d=n - \text{max} \{ \dim_{\mathbb{F}_q}(\mathcal{S} \cap H)\,\,|\,H\,\text{ is an } \mathbb{F}_{q^m}\text{-hyperplane of } \mathbb{F}_{q^m}^k\}.\]

Two $[n,k,d]_{q^m/q}$-systems $\mathcal{S}_1$ and $\mathcal{S}_2$ are said to be {\it equivalent} if there exists an $\mathbb{F}_{q^m}$-isomorphism $\varphi: \F_{q^m}^k \longrightarrow \F_{q^m}^k$ such that $\varphi(\mathcal{S}_1) = \mathcal{S}_2$.

For each $\F_{q^m}$-subspace $V \subseteq \F_{q^m}^k$, we define the {\it weight} of $V$ in $\mathcal{S}$ as the integer $\mathrm{wt}_{\mathcal S}(V)=\dim_{\F_q}(\mathcal{S} \cap V)$. Also, let $\ell,r$ be nonnegative integers such that $\ell<k \leq n$. An $[n,k,d]_{q^m/q}$-system $\mathcal S$ is said to be $(\ell,r)$-\textit{evasive} if for every $\ell$-dimensional $\F_{q^m}$-subspace $W$ of $\F_{q^m}^k$ we have $\wt_{\mathcal S}(W)\leq r$ \cite{bcmT2021}. When $r=\ell$, an $(\ell,\ell)$-evasive $q$-system will be also called $\ell$-\textit{scattered} and if $\ell=1$, we will simply say that it is \textit{scattered}, \cite{cmpz2021}. 

It is easy to see that if $r<\ell,$ there is no $(\ell,r)$-evasive $\F_{q^m}^k$ system. Moreover, if $\mathcal S$ is an $(\ell,r)$-evasive $[n,k,d]_{q^m/q}$-system and $\ell>0$, then $\mathcal S$ is also $(\ell-1,r-1)$-evasive.

Let $\mathfrak U(n,k,d)_{q^m/q}$ denote the set made up of all equivalence classes of $[n,k,d]_{q^m/q}$-systems in $\F_{q^m}^k$, and let $\mathfrak C(n,k,d)_{q^m/q}$ denote the set of equivalence classes of nondegenerate $[n,k,d]_{q^m/q}$-codes. Assume that $[\mathcal C] \in \mathfrak C(n,k,d)_{q^m/q}$ and let $\mathbf{G}=(u_1| \cdots |u_n) \in \mathbb{F}_{q^m}^{k \times n}$, where $u_i \in \mathbb{F}_{q^m}^k$ for $i \in \{1,...,n\}$, such that $\mathcal{C}=\operatorname{rowsp}(\mathbf{G})$. In \cite{randrianarisoa2020} it was determined that the following two maps: 

$$\begin{array}{rccc}\Phi: & \mathfrak C(n,k,d)_{q^m/q} &\longrightarrow &\mathfrak U(n,k,d)_{q^m/q} \\
& [\operatorname{rowsp}(\mathbf{G})] & \longmapsto & [\langle u_1, \ldots, u_n\rangle_{\mathbb{F}_q}] \end{array} $$

$$\begin{array}{rccc}\Psi: & \mathfrak U(n,k,d)_{q^m/q} &\longrightarrow &\mathfrak C(n,k,d)_{q^m/q} \\
& [\langle u_1, \ldots, u_n\rangle_{\mathbb{F}_q}] & \longmapsto & [\operatorname{rowsp}(\mathbf{G})] \end{array}$$
are well defined, and moreover they are the inverses of each other (see also \cite{sv2020}). Hence, the maps $\Phi$ and $\Psi$ define a 1-to-1 correspondence between equivalence classes of $[n,k,d]_{q^m/q}$-codes and equivalence classes of $[n,k,d]_{q^m/q}$-systems. Moreover, in this correspondence that associates a $[n,k,d]_{q^m/q}$ code $\mathcal C$ with an $[n,k,d]_{q^m/q}$ system $\mathcal{S}$, codewords of $\mathcal C$ of rank weight $w$ correspond to hyperplanes $H$ of $\mathbb{F}_{q^m}^k$ with $\dim_{\mathbb{F}_q}(H\cap \mathcal{S})=n-w$.

For a nondegenerate $[n,k]_{q^m/q}$ code $\cC$ and any $\mS\in\Phi([\cC])$, the {\it$s$-th generalized rank weight} of $\cC$ is given by
\[d_s(\cC)=n-\max\{\dimq(\mS\cap W):W\leq\F_{q^m}^k,\ \dimqm(W)=k-s\},\] for $s=1,\ldots,k$. Clearly, $d(\cC)=d_1(\cC)$.

\medskip
In what follows, we will define $\mathcal{M}_{\mathcal{S}}=(\mathcal{S},\rho_{\mathcal{S}})$, as the $q$-matroid whose ground vector space is an $[n,k]_{q^m/q}$ system $\mathcal{S}$ of $\mathbb{F}_{q^m}^k$, and $\rho_{\mathcal{S}}$ is defined as in \eqref{notation}. In this regard, we recall the following result, which will be useful throughout the remainder of this section.

\begin{proposition}[\cite{ajnz2024},  Theorem 1.1]\label{prop:equiv}
    Let $\textbf{G}$ be the generator matrix of a nondegenerate $[n,k]_{q^m/q}$ code, and let $\mathcal{S}$ be any $[n,k]_{q^m/q}$ system associated to it. Then $\mathcal{M}_{\textbf{G}}\cong \mathcal{M}_{\mathcal{S}}.$
\end{proposition}

\begin{proof}
   Let us consider the isomorphism of $\F_q$-vector spaces 
   $$\psi_\mathbf{G} : \vv \in \F_q^n \longrightarrow \vv \mathbf{G}^t \in \mathcal{S}.$$
   For any subspace $U \in \mathscr{L}(\F_q^n),$ let $A^U$ be a matrix whose columns form a basis of $U$. Then,
   \begin{equation*}
       \rho_{\mathcal{S}}(\psi_\mathbf{G}(U))=\dimqm\langle \psi_\mathbf{G}(U)\ra_{\F_{q^m}}=\dimqm(\mathrm{colspan}(\mathbf{G}A^{U}))=\rk_{\F_{q^m}}(\mathbf{G}A^U)=\rho_{\mathbf{G}}(U).
   \end{equation*}
\end{proof}

We say that a $q$-matroid $\mathcal{M}$ of rank $k$ and height $n$ is $\mathbb{F}_{q^m}$-\textit{representable} via an $\F_{q^m}$-linear rank-metric code $\cC$ of length $n$ and dimension $k$ if $\mathcal{M} \simeq \mathcal{M}_\mathbf{G}$, where $\mathbf{G}$ is any generator matrix of $\cC$. 

A classical example of representable $q$-matroid is the so-called \textit{uniform $q$-matroid} $\mathcal U_{k,n}(q)$ of rank $k$ on $\F_q^n$ which is defined over $\F_q^n$ by the rank function
\[
\rho(V)=\min\{k,\dimq(V)\},
\]
for every $V\leq\F_q^n$. More precisely, for $0<k<n$, a matrix $\mathbf G\in\F_{q^m}^{k\times n}$ represents $\mathcal U_{k,n}(q)$ if and only if $\mathbf G$ generates an MRD code~\cite[Example~2.4]{g-lj2025}. Equivalently, in terms of $q$-systems, an $\F_{q^m}$-representation of $\mathcal U_{k,n}(q)$ is given by an $[n,k]_{q^m/q}$ system which is $(k-1)$-scattered. Such a system exists if and only if $m\geq n$~\cite{sv2020}.  Consequently, $\mathcal{U}_{k,n}(q)$ is $\F_{q^m}$-representable if and only if $m\geq n$.


\medskip

Now, denote by $\mathcal{G}_\ell(\F_{q^m}^k)$ the set of all subspaces of $\F_{q^m}^k$ with $\F_{q^m}$-dimension $\ell$ and  consider $T \in \mathcal{G}_\ell(\F_{q^m}^k)$. Let $$\pi_T: \vv \in \F_{q^m}^k \longrightarrow  \bar{\vv}=\vv+T \in \F_{q^m}^k/T$$ be the canonical quotient projection with respect to the subspace $T$. Clearly, the map $\pi_T$ is $\F_{q^m}$-linear and surjective. Moreover, if $\mathcal{S} \subseteq \F_{q^m}^k$ is an $\F_q$-subspace, then $\pi_T(\mathcal{S})\subseteq \F_{q^m}^k/T$ is an $\F_q$-subspace as well.

\begin{lemma}\label{lem:pi}
    Let $\mathcal{S}$ be an $[n,k]_{q^m/q}$-system of $\F_{q^m}^k$ and let  $T \in \mathcal{G}_{\ell}(\F_{q^m}^k)$. Then, for any $\F_{q^m}$-subspace $U$ of $\F_{q^m}^k$ containing $T$,
    \begin{itemize}
       \item [$(i)$]  $\pi_T(\mathcal{S} \cap U)= \pi_T(\mathcal{S}) \cap \pi_T(U)$. 
       \item [$(ii)$] $\dim_{\F_q}(\pi_T(\mathcal{S}) \cap \pi_T(U))= \dim_{\F_q}(\mathcal{S} \cap U) - \dim_{\F_{q}}(\mathcal{S} \cap T)$. 
    \end{itemize}
\end{lemma}
\begin{proof}
    \textit{(i)} We have to prove only that $ \pi_T(\mathcal{S}) \cap \pi_T(U) \subseteq \pi_T(\mathcal{S} \cap U)$. Let $\xx +T \in \pi_T(\mathcal{S}) \cap   \pi_T(U) $, then there exist $\mathbf{s} \in \mathcal{S}$ and $\mathbf{u} \in U$ such that $ \mathbf{s} + T= \xx + T = \uu +T.$ Since $\mathbf{s}-\uu \in T$ and $T \leq U$, we get that $\mathbf{s} \in U$. Hence $\mathbf{x} + T=\mathbf{s} + T \in \pi_T(\mathcal{S} \cap U)$.\\
    \textit{(ii)} Let $ T \leq  U \leq  \F_{q^m}^k$ and consider $\pi_{| \mathcal{S} \cap U}$. It is straightforward to see that $\ker \pi_{T| \mathcal{S} \cap U}= \mathcal{S} \cap T$, then by point $(i)$:
    \begin{equation*}
    \begin{aligned}
        \dimq(\pi_T(\mathcal{S}) \cap \pi_T(U))&=\dimq(\pi_T(\mathcal{S}\cap U))= \dimq(\mathcal{S} \cap U) - \dimq(\ker \pi_{T | \mathcal{S} \cap U})\\
        &=\dimq(\mathcal{S} \cap U) - \dimq(\mathcal{S} \cap T).
    \end{aligned}
    \end{equation*}
    \end{proof}

In what follows, if $\mathcal{S}$ is a $q$-system of $\mathbb{F}_{q^m}^k$ and $T$ is an $\mathbb{F}_{q^m}$-subspace of $\mathbb{F}_{q^m}^k$, we will denote its image under $\pi_T$ by $\bar{\mathcal S}$; that is, $$\bar{\mathcal S} = \pi_T(\mathcal S) \subseteq \F_{q^m}^k/T.$$

\begin{proposition}\label{prop:qevasive}
Let $\mathcal S\subseteq \F_{q^m}^k$ be an $[n,k]_{q^m/q}$-system, and assume that $\mathcal S$ is
$(\ell,r)$-evasive. Let $T \in \mathcal{G}_t(\F_{q^m}^k)$ with $t \leq \ell$ and $\dim_{\F_q}(T\cap \mathcal S)=\tau$.
Then $\bar{\mathcal S}$ is $(\ell-t,\,r-\tau)_q$-evasive and $\dim_{\F_q}(\bar{\mathcal S})=n-\tau$.
\end{proposition}

\begin{proof}
First, since $\langle \mathcal{S} \rangle_{\F_{q^m}} = \F_{q^m}^k$ and the map $\pi_T$ is $\F_{q^m}$-linear, we have $\langle \bar{\mathcal{S}} \rangle_{\F_{q^m}} = \F_{q^m}^k/T.$
Now, let $\overline{U}$ be an $\F_{q^m}$-subspace of $\F_{q^m}^k/T$ of dimension $\ell-t$. Then there exists an $\F_{q^m}$-subspace $U \in \mathcal{G}_{\ell}(\F_{q^m}^k)$ of $\F_{q^m}^k$ containing $T$ such that $\overline{U} = U/T$. By the assumptions and Lemma~\ref{lem:pi}, we get
\begin{equation*}
    \dimq(\bar{\mathcal{S}} \cap \overline{U})
    =
    \dimq(\mathcal{S} \cap U)-\dimq( \mathcal{S} \cap T)
    \leq r - \tau.
\end{equation*}
Hence, the assertion follows.
\end{proof}

\section{$\mathbb{F}_{q^m}$-linear near-MRD codes and their $q$-matroids}\label{sec:near-MRD codes and their $q$-matroids}

In this section, we collect some properties of a class of rank-metric codes known as near-MRD codes and on their associated $q$-systems. An $[n,k]_{q^m/q}$ code $\cC$ is called \textit{near-MRD} if
$d(\cC)=n-k$ and
$d_s(\cC)=n-k+s$ for every $2\leq s\leq k$. For a nondegenerate $[n,k,d]_{q^m/q}$ code $\cC$, it is near-MRD if and only if $d(\cC)+d(\cC^\perp)=n$, or equivalently, if any $\mS\in\Phi([\cC])$ is $(k-2)$-scattered, not $(k-1)$-scattered, and $(k-1,k)$-evasive~\cite[Proposition~5.3]{mnt23}.

Note explicitly that since $\mS$ is not $(k-1)$-scattered and $(k-1,k)$-evasive, there exists an $\F_{q^m}$-vector space $H$ such that $\dimqm H=k-1$ and $\wt_\mS(H)=k$. 

\begin{proposition}[\cite{mnt23}, Theorem 5.5]\label{upperbound-near}
Let $\mathcal{C}$ be an $[n,k]_{q^m/q}$ near-MRD code. Then, $m \ge k$ and
\[
n \le
\begin{cases}
m+2 & \text{if } m = 2k-2,\\
m+1 & \text{if } m \ne 2k-2 .
\end{cases}
\]
Moreover, $\mathcal{C}$ is a $[2k,k,k]_{q^{2k-2}/q}$ near-MRD code if and only if any 
$\mS \in \Phi([\mathcal{C}])$ is a maximum $(k-2)$-scattered subspace.
\end{proposition}
We call an $[n,k]_{q^m/q}$ near-MRD code \textit{optimal} if its length attains the bound in Proposition~\ref{upperbound-near}, that is, $n=m+2$ if $m=2k-2$ and $n=m+1$ otherwise.

The following result gives a general construction of near-MRD codes
of smaller length from a given near-MRD code.

\begin{proposition}\label{prop:subsystem-near-MRD}
Let $\mathcal{U}$ be an $[N,k]_{q^m/q}$ system associated with a near-MRD
code. Then, for every $k+1\leq n\leq N$, there exists an
$[n,k]_{q^m/q}$ system $\mathcal{S}\subseteq\mathcal{U}$ associated
with a near-MRD code.
\end{proposition}

\begin{proof}
Since $\mathcal{U}$ is near-MRD, there exists an $\F_{q^m}$-hyperplane
$H$ such that $\wt_{\mathcal{U}}(H)=k$. Set
$R=\mathcal{U}\cap H$, then, since $\mathcal{U}$ is $(k-2)$-scattered,
$\langle R\rangle_{\F_{q^m}}=H$. Choose $\mathbf{v}\in\mathcal{U}\setminus H$, then, since
$R=\mathcal{U}\cap H$, we have $\vv\notin R$, and hence
$\dimq(R+\langle \vv\rangle_{\F_q})=k+1$. As
$\dimq(\mathcal{U})=N$, for every $n$ with $k+1\leq n\leq N$,
the $\F_q$-basis of $R+\langle \vv\rangle_{\F_q}$ can be extended to an
$n$-dimensional $\F_q$-subspace $\mathcal{S}$ of $\mathcal{U}$.
Thus
$R+\langle \vv\rangle_{\F_q}\subseteq\mathcal{S}\subseteq\mathcal{U}$
and $\dimq(\mathcal{S})=n$.
Moreover, since $\langle R \rangle_{\F_{q^m}}=H$ and $\vv\notin H$, we have
$\langle R,\vv\rangle_{\F_{q^m}}=\F_{q^m}^k$. Therefore,
$\langle\mathcal{S}\rangle_{\F_{q^m}}=\F_{q^m}^k$, and hence
$\mathcal{S}$ is an $[n,k]_{q^m/q}$ system.

 Since
$\mathcal{S}\subseteq\mathcal{U}$, it is $(k-2)$-scattered and
$(k-1,k)$-evasive. Moreover, $R\subseteq\mathcal{S}\cap H$ and
$\dimq(R)=k$. Since $R \subseteq \mathcal{S} \cap H \subseteq\mathcal{U} \cap H = R$, we have that $\dimq(\mathcal{S}\cap H)=k$. Thus
$\mathcal{S}$ is not $(k-1)$-scattered and  $\mathcal{S}$ is
associated with an $[n,k]_{q^m/q}$ near-MRD code.
\end{proof}

\begin{example}[\cite{mnt23}, Proposition 5.6]\label{ex:nearMRD}
 Assume that $m \geq k$. Then,  the set $$\mathcal S := \{(\alpha+\lambda, \alpha^q, \ldots,\alpha^{q^{k-1}}) \,:\, \alpha \in \mathbb F_{q^m},\, \lambda \in \mathbb{F}_q\}$$ is an $[m+1,k]_{q^m/q}$ system, which is $(k-2)$-scattered and $(k-1,k)$-evasive.

Hence, by definition, any code $\cC\in\Psi([\mathcal S])$ is an $[m+1,k]_{q^m/q}$ near-MRD code.  Moreover, by Proposition \ref{upperbound-near}, this is the largest possible length of a  code for given values of $k,m$ (and $q$), provided $m\neq 2k-2$.
\end{example}

By Proposition~\ref{prop:subsystem-near-MRD}, the existence of the
above $[m+1,k]_{q^m/q}$ near-MRD system implies the existence of an
$[n,k]_{q^m/q}$ near-MRD system for every $k+1\leq n\leq m+1$.

\begin{lemma} \label{lem:hyp-cyc}
 Let $\cC$ be a nondegenerate $[n,k]_{q^m/q}$ near-MRD code and let $\mS \in \Phi([\cC])$. Then, there is a bijection between the hyperplanes $H$ of $\F_{q^m}^k$ such that $\wt_\mathcal{S}(H)=k$, and the $k$-dimensional $\F_q$-subspaces $W \subseteq  \mS$ such that $\langle W\rangle_{\F_{q^m}}$ is a hyperplane of $\F_{q^m}^k$.  
\end{lemma} 
\begin{proof} Let $\mS \in \Phi([\cC]) \subseteq \F_{q^m}^{k}$ and recall the $s$-th generalized rank-weight
\[
d_s(\cC)=n-\max\bigl\{\dimq(\mS\cap \Pi): \Pi\leq \F_{q^m}^{k}, \dimqm(\Pi)=k-s\bigr\},
\qquad s=1,\dots,k.
\]
Since $\cC$ is near-MRD, we have $d_s(\cC)= n-k+s$ for every $s=2,\ldots,k$.

Now, define the following sets:
$$\mathcal{H}= \Bigl\{H \in \mathcal{G}_{k-1}(\F_{q^m}^k) \colon \dimq(\mS\cap H)=k\Bigr\}$$ and $$\mathcal{W}=
\Bigl\{W \subseteq  \mS:\ \dimq(W)=k,\ \langle W\rangle_{\F_{q^m}} \in \mathcal{G}_{k-1}(\F_{q^m}^{k})\Bigr\}.$$ Let $H \in \mathcal{H}$ and set $W:=\mS\cap H$.
Clearly, $W$ is a $\F_q$-subspace of $\mS$ and $\dimq(W)=k$. It remains to prove that
$\langle W\rangle_{\F_{q^m}}$ is a hyperplane. Assume by contradiction that $\dim_{\F_{q^m}}\bigl(\langle W\rangle_{\F_{q^m}}\bigr)\le k-2$ and let $U$ be a $\F_{q^m}$-subspace of dimension $k-2$ containing $\langle W \rangle_{\F_{q^m}}$. Then $\dimq(\mS\cap U)\ge \dimq(W)=k.$ On the other hand,
\[
d_2(\cC)\le n-\dim_q(\mS\cap U)\le n-k,
\]
which contradicts $d_2(\cC)= n-k+2$. Hence, $\dim_{\F_{q^m}}\bigl(\langle W\rangle_{\F_{q^m}}\bigr)=k-1$, 
so $\langle W\rangle_{\F_{q^m}}$ is a hyperplane. Thus the map $\varphi: H \in \mathcal{H} \mapsto \mS \cap H \in \mathcal{W}$ is well defined.

Next, let $W \in \mathcal{W}$ and set $H:=\langle W\rangle_{\F_{q^m}}$. Since $W\subseteq \mS\cap H$, we have $\dimq(\mS\cap H)\ge k.$
On the other hand, since
\[
\max\bigl\{\dimq(\mS\cap H'): H' \text{ hyperplane of } \F_{q^m}^{k}\bigr\}=k,
\]
we have $\dimq(\mS\cap H)=k$, and so the map $\psi: W \in \mathcal{W} \mapsto \langle W \rangle_{\mathbb{F}_{q^m}} \in \mathcal{H}$ is well defined.

Finally, we prove that $\varphi$ and $\psi$ are inverse to each other.
Let $H \in \mathcal{H}$, then $\psi(\varphi(H))
=\langle \mS\cap H\rangle_{\F_{q^m}}.$
By the first part of the proof, $\langle \mS\cap H\rangle_{\F_{q^m}}$ is a hyperplane; since it is contained in $H$, it must coincide with $H$. Hence $\psi(\varphi(H))=H$.
Now, let $W \in \mathcal{W}$, then $H:=\langle W\rangle_{\F_{q^m}}$ is a hyperplane and $W\subseteq \mS\cap H.$
But, as shown above, $\dimq(\mS\cap H)=k=\dimq(W)$, so necessarily $\mS\cap H=W.$
Thus $\varphi(\psi(W))=\varphi(H)=\mS\cap H=W.$
Therefore $\varphi$ is a bijection, with inverse $\psi$.
\end{proof} 

Let $\cC$ be a nondegenerate $[n,k]_{q^m/q}$ near-MRD code and let $\mathcal{S} \in \Phi([\cC])$. In the following, our aim is to prove an upper bound on the number of hyperplanes $H$ of $\F_{q^m}^k$ containing a fixed $(k-2)$-dimensional $\F_{q^m}$-vector space such that $\wt_\mathcal{S}(H)=h \leq k$.

Let $T \in \mathcal{G}_{k-2}(\F_{q^m}^k)$ and set  $\overline{V}:=\pi_T(\F_{q^m}^k)$, we have that $\dimqm \overline{V}=2$. Let $H$ be any hyperplane of $\F_{q^m}^k$ containing $T$, then $\overline{H}=\pi_T(H)$ is a hyperplane of the $2$-dimensional $\mathbb F_{q^m}$-space $\overline{V}$, that is, a $1$-dimensional $\mathbb F_{q^m}$-subspace. By Proposition \ref{prop:qevasive}, if $\wt_{\mS}(T)=\dimq(\mathcal{S} \cap T)=\tau$ (note that $\tau$ must be at most $k-2$), the $\F_{q}$-linear vector space $\overline{\mathcal{S}} \subseteq  \overline{V}$ determines an $\F_q$-linear space of $\overline{V}$ of rank $n':=n - \tau$  such that any $1$-dimensional subspace $\overline{H}$ has weight with respect to $\overline{\mathcal{S}}$ at most $k':=k-\tau$. Therefore, $\mathcal{\bar{S}}$ is an $[n-\tau,2]_{q^m/q}$ system with minimum distance at least $n-k$. Moreover, if $m \geq 3$, the code $\bar{\cC} \in \Psi([\bar{\mathcal{S}}])$ is not a simplex. Indeed, by~\cite[Corollary~3.17]{linearcutting},  if $\overline{\cC}$ were a one-weight code, $n-\tau=2m$ and $d(\overline{\cC})=m$. On the other hand, since $\cC$ is near-MRD,
Proposition~\ref{upperbound-near} gives $n\leq m+2$. Since $\tau\geq 0$, we obtain $2m=n-\tau\leq m+2$, and therefore $m\leq 2$, getting a contradiction.\\
In the case $m=2$, this can occur. Indeed, since $2\leq k\leq m$, necessarily $k=2$. Choosing $T=\{\mathbf{0}\}$ and
$\tau=0$, we have $n-\tau=2m=4$. Hence
$\overline{\cC}$ may be a simplex rank-metric code with parameters $[4,2,2]_{q^2/q}$. Such a code is one-weight and is also near-MRD.

\begin{proposition}\label{general}
    Let $\mathcal{\cC}$ be a nondegenerate $[n,k]_{q^m/q}$ near-MRD code  and let $\mathcal{S} \in \Phi([\cC])$. Suppose that $T$ is a $(k-2)$-dimensional $\F_{q^m}$-vector space such that $\dimq(\mathcal{S} \cap T)=\tau$. Then, for any $ 0 \leq \tau < h \leq k$,
\begin{equation} \label{sizeH}
    \# \left \{ H \in \mathcal{G}_{k-1}(\F_{q^m}^k) \colon T \leq H \text{ and } \dimq(H \cap \mathcal{S}) =h  \right \} \leq  \left \lfloor \frac{q^{n-\tau}-1}{q^{h-\tau}-1}  \right \rfloor.
\end{equation}

\end{proposition}
\begin{proof}
As before, setting $\overline{V}:=\F_{q^m}^k/T$, we have that $\dimqm\overline{V}=2$. So, the $1$-dimensional $\F_{q^m}$-subspaces of $V/T$ are precisely the points of $\mathrm{PG}(\overline{V},\F_{q^m})=\PG(1,q^m)$ and each hyperplane $H$ containing $T$ corresponds to the point $\overline{H}\in\PG(1,q^m)$. As seen before, the $\F_{q}$-linear  $\overline{\mathcal{S}} \subseteq \overline{V}$ is an $[n-\tau,2]_{q^m/q}$ where $\wt_\mathcal{S}(T)=\tau$. The $\F_q$-vector space $\mathcal{\bar{S}} \subseteq \overline{V}$ determines a linear set $L_{\overline{\mathcal{S}}}$ of $\PG(1,q^m)$ of rank $n':=n - \tau$  such that  $|L_{\overline{\mathcal{S}}}| \geq 2$ and  any point has weight (with respect to $\overline{\mathcal{S}}$) at most $k':=k-\tau$.\\
Thus, counting the size of the set in~\eqref{sizeH} is equivalent to counting the points of
$L_{\overline{\mathcal S}}$ having weight $h'=h-\tau$ (w.r.t. $\overline{\mathcal S}$).
For each integer $ 1 \leq i\le k'$, let $N_i$ denote the number of points of weight $i$ in $L_{\mathcal{\bar{S}}}$ (w.r.t. $\overline{\mathcal{S}}$ ). By \cite[Proposition 2.2]{Polv}, we have that
$$\sum_{i=1}^{k'}N_i(q^i-1)=q^{n'}-1$$
and hence, 
$$N_{h'}(q^{h'}-1)\le \sum_{i=1}^{k'}N_i(q^i-1)=q^{n'}-1.$$
This leads to the result.
\end{proof}

\subsection{The $q$-matroid of a near-MRD code}

In \cite{gj22}, the authors introduced the $q$-analogue of a special class of so-called paving matroids. More precisely, a $q$-matroid $\mathcal M$ is called \textit{paving} if all its circuits $C$ satisfy $\dimq(C)\geq \rho(\mathcal M)$. 

Let $n \ge 2$ and fix an integer $1 \leq k\leq n$. Let $\mathcal{D}_k$ be a collection of $k$-dimensional subspaces of $\F_q^n$ such that
\begin{equation}\label{def:paving}
\dimq(U \cap V) \le k-2 \quad \text{for all distinct } U,V \in \mathcal{D}_k.
\end{equation}
The following result was then provided.

\begin{proposition} [\cite{gj22}, Proposition 4.6] 
Let $\rho: \mathscr{L}(\F_q^n)\longrightarrow \mathbb{N}_0$, be the map defined in the following way
\begin{align*}
\rho(U) =
\left\{
\begin{aligned}
  &k-1 &&\text{if}~U\in\mathcal{D}_k, \\
  &\min\{\dimq U,k\} &&\text{if}~U\notin \mathcal{D}_k. \\
\end{aligned}
\right.
\end{align*}
If $\mathcal{D}_k$ satisfies the Property \eqref{def:paving}, then $(\F_q^n,\rho)$ is a $q$-matroid.
\end{proposition}

In the following, we will denote the paving $q$-matroid $(\F_q^n,\rho)$ described in the above statement by $\mathcal{M}_{n,q}(\mathcal{D}_k)$. Actually, the condition expressed in \eqref{def:paving}
is also necessary for $\rho$ to define a $q$-matroid. Indeed, if $\dimq ( U \cap V)=k-1$, then
\[
\rho(U+V)+\rho(U\cap V)=k+(k-1)=2k-1,
\]
whereas
\[
\rho(U)+\rho(V)=(k-1)+(k-1)=2k-2.
\]
This contradicts the submodularity axiom. Note that if $\mathcal{D}_k$ is the empty set, then $\mathcal{M}_{n,q}(\mathcal{D}_k)=\mathcal{U}_{k,n}(q)$. Moreover, if $k=1$ then $|\mathcal{D}_k| \leq 1$. If $|\mathcal{D}_k|=1$, we have a $q$-matroid of rank $1$ with one loop.

The next proposition aims to describe the cyclic flats of $\mathcal{M}_{n,q}(\mathcal{D}_k)$. 

\begin{proposition} \label{paving-matroid}
Let $n \geq 2$ and $1 \leq k \leq n$ and consider the $q$-matroid $\mathcal{M}:=\mathcal{M}_{n,q}(\mathcal{D}_k)$. 

\begin{enumerate}
 \item If $1\leq k<n$, then 
    \[
    \mathscr{Z}(\mathcal{M})=
    \{\langle \mathbf{0} \rangle \}\cup \mathcal D_k \cup
    \begin{cases}
    \{\F_q^n\}, & \text{if } n\ge k+2,\\
    \{\F_q^n\}, & \text{if } n=k+1 \text{ and } \mathcal D_k=\emptyset,\\
    \emptyset, & \text{if } n=k+1 \text{ and } \mathcal D_k\neq\emptyset.
    \end{cases}
    \]

    \item If $k=n$, then $\mathcal D_k \subseteq  \{\F_q^n\}$, and
    \[
    \mathscr{Z}(\mathcal{M})=
    \begin{cases}
    \{ \langle \mathbf{0} \rangle \}, & \text{if } \mathcal D_k=\emptyset,\\
    \{\langle \mathbf{0} \rangle ,\F_q^n\}, & \text{if } \mathcal D_k=\{\F_q^n\}.
    \end{cases}
    \]
\end{enumerate}
\end{proposition}

\begin{proof}
Clearly, we have that $\{\mathbf{0}\}$ is a cyclic flat for any $ 1 \leq k \leq n$. Consider the case $1\le k<n$ and let $D\in \mathcal D_k$. Then $\dimq D=k$ and $\rho(D)=k-1$. Every hyperplane $H$ of $D$ is not in $\mathcal{D}_k$  and $\rho(H)=k-1$.
Moreover, for every $(k+1)$-dimensional vector space $X>D$, we have that $\rho(X)=k>\rho(D)$. Thus $D$ is both cyclic and  flat, so $D$ is a cyclic flat.

If $D$ is a $k$-subspace not belonging to $\mathcal D_k$, then $\rho(D)=k$, and every subspace $X > D$ has rank $k$ as well, so $D$ is not a flat. 

Next, let $D \in \mathscr{L}(\F_q^n)$ with $1\le \dimq D <k$. Then $\rho(D)=\dimq D$.  For any hyperplane $H$ of $D$, we have that $\rho(H)=\dimq D -1 $,  so $D$ is not cyclic. 
Similarly, if $k<\dimq D<n$, then $\rho(D)=k$ and every subspace $X>D$ still has rank $k$, and hence $D$ is not a flat. Therefore, among the proper nonzero subspaces, the only cyclic flats are precisely the members of $\mathcal D_k$.

Since $\F_{q}^n$ is flat, it remains to determine when $\F_q^n$ is cyclic. Since $\rho(\F_q^n)=k$, this happens exactly when every hyperplane $H<\F_q^n$ satisfies $\rho(H)=k$.

If $n\ge k+2$, then every hyperplane has dimension $n-1\ge k+1$, so indeed $\rho(H)=k$ for every hyperplane $H$, and therefore $\F_q^n$ is cyclic.

If $n=k+1$, then the hyperplanes of $\F_q^n$ are exactly its $k$-subspaces. Now, a hyperplane $H$ has rank $k-1$ precisely when $H\in\mathcal D_k$. Hence $\F_q^n$ is cyclic if and only if no hyperplane belongs to $\mathcal D_k$, that is, if and only if $\mathcal D_k=\emptyset$.

Finally, suppose $k=n$. Then, it is straightforward to see that $\mathcal D_k=\emptyset$ or $\mathcal D_k=\{\F_q^n\}$. If $\mathcal D_k=\emptyset$, then $\rho(U)=\dimq U$ for every $U \in \mathscr{L}(\F_{q}^n)$, so the only cyclic flat is $\{\mathbf{0}\}$. If $\mathcal D_k=\{\F_q^n\}$, then $\rho(\F_q^n)=n-1$, while every hyperplane $H$ of $\F_q^n$ has rank $n-1$, so $\F_q^n$ is cyclic and  flat. Thus in this case the cyclic flats are exactly $\{\mathbf{0}\}$ and $\F_q^n$.
\end{proof}

\begin{proposition}\label{prop:paving-matroid}
Let $n\ge 2$ and $1\le k< n$. Let $\mathcal M=(\F_q^n,\rho)$ be a $q$-matroid of rank $k$, and define
\[
\mathcal D_k:=\{D\le \F_q^n : \dim_{\F_q}D=k \text{ and } \rho(D)=k-1\}.
\]
Assume that the cyclic flats of $\mathcal M$ are exactly those listed in Proposition~\ref{paving-matroid}. Then, $\mathcal{M}=\mathcal {M}_{n,q}(\mathcal D_k)$.
\end{proposition}

\begin{proof} In the case where $\mathcal{D}_k= \emptyset$, the only cyclic flats of $\mathcal{M}$ are the null-space and $\F_{q}^n$, then $\mathcal{M}=\mathcal{U}_{k,n}(q)$ as proved in \cite[Proposition 2.30]{alfbyrne}. Now suppose $\mathcal{D}_k \neq \emptyset$, we shall show that
\[
\rho(U)=
\begin{cases}
\dimq U & \text{if }\dimq U<k,\\
k-1 & \text{if }\dimq U=k \text{ and } U\in\mathcal D_k,\\
k & \text{otherwise.}
\end{cases}
\]

\noindent \textit{Case 1:} Let $U\in\mathcal D_k$ and consider $V<U$.
If $V$ is a hyperplane of $U$, then, since $U$ is cyclic, $\rho(V)=\rho(U)=k-1.$
Hence, by Lemma~\ref{lesskkk}, every proper subspace $T< U$ has rank equal to its dimension. 
Now let $W>U$. Since $U$ is a flat, then every subspace $W$ such that $W > U$ has rank strictly larger than $\rho(U)$, i.e. $\rho(W)>\rho(U)=k-1$.
As $\mathcal M$ has rank $k$, then
$\rho(W)=k$.

Therefore, if $U\in\mathcal D_k$, then every proper subspace of $U$ is independent and every subspace properly containing $U$ has rank $k$.

\noindent
\textit{Case 2:} Let $U\notin\mathcal D_k$ with $\dimq U=k$.
We claim that $\rho(U)=k$. Indeed, assume by contradiction that $\rho(U)<k$. Since $U\notin\mathcal D_k$, we must have that $\rho(U)\le k-2$.
Thus $U$ is a dependent not minimal space, and hence it contains a proper circuit $C<U$ (i.e, such that $\dimq C\le k-1$).
By Lemma~\ref{alfbyrne-lemma228}, there exists a cyclic flat $X$ such that $\{\mathbf{0}\}<C\le X<\F_q^n$
and $\dimq C=\rho(X)+1$.
Hence
\[
\rho(X)=\dimq C-1\le k-2.
\]

Now, by the description of the cyclic flats in Proposition~\ref{paving-matroid}, every proper nonzero cyclic flat is an element of $\mathcal D_k$, and therefore has rank $k-1$. This is a contradiction. Hence $\rho(U)=k$.

\noindent
\textit{Case 3:} Let $U \in \mathscr{L}(\F_q^n)$ with $\dimq U<k$.
We show that $\rho(U)=\dimq U$.

Certainly, $\rho(U)\le \dimq U$. Suppose by contradiction that $\rho(U)<\dimq U$. Then again $U$ is dependent. Therefore, it contains a circuit $C$, and since $\dimq U<k$, we have $\dimq C\le \dimq U<k$.

Again by Lemma~\ref{alfbyrne-lemma228}, there exists a cyclic flat $X$ such that $\{\mathbf{0}\}<C\le X<\F_q^n$
and $\dimq C=\rho(X)+1$.
Thus
\[
\rho(X)=\dimq C-1<k-1.
\]
However, every proper nonzero cyclic flat belongs to $\mathcal D_k$, hence has rank $k-1$. This is impossible. Therefore $U$ is independent, and so
$\rho(U)=\dimq U.$

\noindent
\textit{Case 4:} Let $U \in \mathscr{L}(\F_q^n)$ with $\dimq U>k$.
We show that $\rho(U)=k$.
Since $\mathcal M$ has rank $k$, we have $\rho(U)\le k$. Suppose by contradiction that $\rho(U)\le k-1$. Then $U$ is dependent, so it contains a circuit $C\le U$. By Lemma~\ref{alfbyrne-lemma228}, there exists a cyclic flat $X$ such that $C\le X$
 and $\dimq C=\rho(X)+1.$
By Proposition~\ref{paving-matroid}, the cyclic flat $X$ is either an element of $\mathcal D_k$ or $X=\F_q^n$.

If $X\in\mathcal D_k$, then $\rho(X)=k-1$, and therefore
\[
\dimq C=\rho(X)+1=k.
\]
Since $C$ is a circuit, it is minimal dependent and hence we have $\rho(C)=\dimq C-1=k-1.$
Thus $C\in\mathcal D_k$. As $\dimq U>k=\dimq C$, we have $C<U$. Since $C$ is a flat, it follows that
\[
\rho(U)>\rho(C)=k-1,
\]
hence $\rho(U)=k$, a contradiction.

If $X=\F_q^n$, then $\rho(X)=k$, so $\dimq C=\rho(X)+1=k+1$.
Again, since $C$ is a circuit, $\rho(C)=\dimq C-1=k.$
By monotonicity,
\[
k=\rho(C)\le \rho(U),
\]
contradicting the assumption that $\rho(U)\le k-1$. Therefore $\rho(U)=k$.

\noindent Combining \textit{Cases~1--4}, we get the result and $\mathcal M=\mathcal M_{n,q}(\mathcal D_k)$.
\end{proof}

Although in \cite{gj22} the authors provided some examples of $q$-matroids belonging to the class $\mathcal{M}_{n,q}(\mathcal{D}_k)$ that are not representable, the problem of determining whether in general such $q$-matroids are non-representable, is left open.

In what follows, we show that the $q$-matroid associated with an $\F_{q^m}$-linear near-MRD code belongs to the class of $q$-matroids described above. Hence, for suitable choices of the parameters $n$, $q$ and $k$, near-MRD codes provide examples of paving $q$-matroids which are $\F_{q^m}$-representable.

\begin{theorem}\label{th:near-matroid}
    Let $\mathcal{C}$ be an $[n,k]_{q^m/q}$ near-MRD code with generator matrix $\textbf{G} \in \F_{q^m}^{k \times n}$ and consider the $q$-matroid  $\mathcal{M}_{\textbf{G}}=(\F_{q}^n,\rho_\mathbf{G})$ represented by $\mathbf{G}$. Then, 
\begin{align}\label{near-rank-function}
\rho_\mathbf{G}(U) =
\left\{
\begin{aligned}
  &k &&\text{if}~\dimq U \geq k+1, \\
  &\dimq (U) &&\text{if}~\dimq U \leq k-1 \\
\end{aligned}
\right.
\end{align}
and $\rho_\mathbf{G}(U)\in\{k,k-1\}$ if $\dimq(U)=k$.
\end{theorem}

 \begin{proof}
Let $\textbf{G}= (\textbf{g}_1\,|\,  \textbf{g}_2|\,  \ldots\,|\, \textbf{g}_n )$
be the generator matrix of $\mathcal{C}$, where 

$$\mathbf{g}_i=\begin{pmatrix} g_{1i} \\
g_{2i} \\
\vdots \\
g_{ki}
\end{pmatrix} \in \mathbb{F}_{q^m}^{k \times 1}$$ denotes the $i$-th column of $\mathbf{G}$, $i=1,2,\ldots,n$.
 
Also, let $U$ be a subspace of $\mathbb{F}_q^n$. By \cite[Lemma 4.2]{gj22}, the rank function $\rho_\mathbf{G}(U)=\rk_{\F_{q^m}}(\textbf{G}A^U)$, where $A^U$ is a matrix belonging to $\F_q^{n \times \dimq(U)}$ whose columns span $U$. By  \cite[Proposition 6.2]{gjhr20}, we obtain that $\rho_\mathbf{G}(U)=k$ for any subspace $U$ of $\F_q^n$ with  $\dimq (U)\geq k+1$. 

Let us consider $\cC^\perp$, the dual of $\cC$, which is an $[n,n-k]_{q^m/q}$ rank-metric code with minimum distance $k$. Since $\mathbf{G}$ is a parity-check matrix of $\mathcal{C}^{\perp}$ and applying Theorem \ref{gabidulin} to $\cC^\perp$, we have that
$\rho_\mathbf{G}(U)=\rk_{\F_{q^m}}(\textbf{G}A^U)=k-1$ for any $U \leq \F_q^n$ with $\dimq(U)=k-1$. 
Moreover, there exists a $k$-dimensional subspace $\widetilde{U} \leq \F_q^n$ such that $\rho_\mathbf{G}(\widetilde{U})=\rk_{\F_{q^m}}(\textbf{G}A^ {\widetilde{U}})=k-1$. Therefore, based on the monotonicity of the rank function, we can conclude that $\rho_\mathbf{G}(U)\in \{k,k-1\}$ whenever $\dimq(U)=k$. 

Finally, it follows from Lemma  \ref{lesskkk}  that $\rho_\mathbf{G}(U)=\dimq U$ for any subspace of $\F_q^n$ with $\dimq (U)\leq k-1$. 

Now, since $\rk_{\F_{q^m}}(\textbf{G})=k$, there is a $k\times k$ sub-matrix of $\textbf{G},$ say $$\textbf{G}'=(\textbf{g}_{i_1}|\textbf{g}_{i_2}|\ldots|\textbf{g}_{i_k}) \,\, \text{ where } 1\leq i_1<i_2<\ldots<i_k\leq n$$  such that $\rk_{\F_{q^m}}(\textbf{G}')=k$. 
Let $A=(a_{ij})$ be an $n\times k$ matrix over $\mathbb{F}_q$ whose entries are defined in the following way:
\[
a_{i_jj}=
\begin{cases}
1 & \text{if } j=1,2,\ldots,k,\\[4pt]
0 & \text{otherwise,}
\end{cases}
\]
and let $U$=colspan$(A)$. Since ${\rk}_{\mathbb{F}_{q}}(A)=k$ the subspace $U \leq \mathbb{F}_{q}^n$ has dimension $k$.
It is straightforward to see that $\rho_\mathbf{G}(U)={\rk}_{\mathbb{F}_{q^m}}(\textbf{G}A)={\rk}_{\mathbb{F}_{q^m}}(\textbf{G}')=k$.

This concludes the proof. 
\end{proof}

With the same notation as in the statement of the previous result, let $\cC$ be an $[n,k]_{q^m/q}$ near-MRD code, $\mathbf{G}$ be a generator matrix of $\cC$ and $\mathcal{M}_\mathbf{G}=(\F_{q}^n, \rho_{\mathbf{G}})$ be the $q$-matroid represented by $\mathbf{G}$. Also, consider the set 
\[\mathcal{D}_k(\cC)=\{U \in \mathscr{L}(\F_q^n) \colon \dimq U=k \textnormal{ and } \rho_{\mathbf{G}}(U)=k-1\}\] which, as we have seen, is not empty.
 Then, it is straightforward to see that for any distinct $U,V \in \mathcal{D}_k(\cC)$, we have $\dimq (U \cap V) \leq k-2$, and the rank function in \eqref{near-rank-function} can be re-written as 
\begin{equation}\label{eq:rank-near}
    \rho_\mathbf{G}(U)=
    \begin{cases}
        k-1 & \text{ if } U \in \mathcal{D}_k(\cC)\\
        \min\{\dimq U , k\} & \text{ if } U \not \in \mathcal{D}_k(\cC)
    \end{cases}.
\end{equation}
Hence, $\mathcal{M}_{\mathbf{G}}=\mathcal{M}_{n,q}(\mathcal{D}_k)$ where $\mathcal{D}_k:=\mathcal{D}_k(\cC)$.

Let $1 \le k < n$, let $\cC$ be an $[n,k]_{q^m/q}$ near-MRD code, and let $\mS \in \Phi([\cC])$. By Proposition \ref{prop:equiv}, we have $\mathcal{M}_{n,q}(\mathcal{D}_k) \cong \mathcal{M}_{\mathbf G} \cong \mathcal{M}_{\mathcal S}$. Since equivalences preserve cyclic flats, it follows that the set of cyclic flats of $\mathcal{M}_{\mathbf G}$ are in bijection with the set of cyclic flats of $\mathcal{M}_{\mathcal S}$. Then, we have that

\
    \[
    \mathscr{Z}(\mathcal{M}_{\mS})=
    \{\mathbf{0}\}\cup \mathcal D_k(\mS) \cup
    \begin{cases}
    \{\mS\}, & \text{if } n\ge k+2,\\
    \emptyset, & \text{if } n=k+1 \text{ and } \mathcal D_k(\mS)\neq \emptyset.
    \end{cases}
    \]
where 
$$\mathcal{D}_k(\mS)=\{Z \leq \mS \colon \dimq Z= k \text{ and } \dimqm \langle Z \rangle_{\F_{q^m}} = k-1\}.$$

\begin{remark}
Let $\cC$ be an $[n,k]_{q^m/q}$  near-MRD code and $\mS \in \Phi([\cC])$. Clearly, since $\mathcal{M}_{\mS}\cong \mathcal{M}_{n,q}(\mathcal{D}_k(\mS))$, any two distinct elements $X,Y \in \mathcal{D}_k(\mS)$ satisfy that $\dimq(X\cap Y)\le k-2$. This property can be obtained also by the evasiveness properties of the $q$-system $\mS$.
Indeed, if $X,Y \in \mathcal{D}_k(\mathcal S)$ are distinct, then by Lemma \ref{lem:hyp-cyc} the subspaces $H_1=\langle X\rangle_{\F_{q^m}}$ and $H_2=\langle Y\rangle_{\F_{q^m}}$
are distinct hyperplanes of $\F_{q^m}^k$. Since $\mathcal S$ is $(k-2)$-scattered, we obtain
\[
\dimq(X\cap Y)\le \dimq\bigl(\mathcal S\cap (H_1\cap H_2)\bigr)\le k-2.
\]
\end{remark}
\medskip
Let $\cC$ be an $[n,k]_{q^m/q}$ near-MRD code, and let $\mS \in \Phi([\cC])$. From now on, the $q$-matroid associated with $\cC$, equivalently $\mathcal{M}_{\mS}$, will be called a \textit{near-uniform} $q$-matroid and will be denoted by the symbol $\mathcal{N}_{k,n,q}(\mathcal{C})$. We conclude this section proving an upper bound on the size of $\mathcal{D}_k(\mS)$.

\begin{theorem}\label{thm:bound-Dk-near-MRD}
Let $\cC$ be an $[n,k]_{q^m/q}$ near-MRD code and let $\mS \in \Phi([\cC])$. Then,

\[
|\mathcal D_k(\mS)|
\le
\frac{\qbinom{n}{k-1}}{\qbinom{k}{1}}.
\]
\end{theorem}

\begin{proof}
Fix a $(k-2)$-dimensional $\mathbb F_q$-subspace $A\le \mS$, and set $T:=\langle A\rangle_{\mathbb F_{q^m}}$. By \eqref{eq:rank-near}, we have that $\dimqm T = \rho_{\mS}(A)=k-2$. Since $\rho_\mS(A)=\dimqm T$ and since $\mS$ is $(k-2)$-scattered, we also have that 
$\wt_\mS(T)=k-2$.
Now let $Z\in\mathcal D_k(\mS)$ and  let $H:=\la Z \ra_{\F_{q^m}} \in \mathcal{G}_{k-1}(\F_{q^m}^k)$. By Lemma~\ref{lem:hyp-cyc}, we have that
\[
\#\bigl\{
Z\in\mathcal D_k(\mS): A\subseteq Z
\bigr\}
=
\#\bigl\{
H\in \mathcal{G}_{k-1}(\mathbb F_{q^m}^k):
T\le H,\;
\dimq(\mS\cap H)=k
\bigr\}.
\]
Applying Proposition~\ref{general} with $\tau=k-2$ and $h=k$, we obtain
\[
\#\bigl\{
Z\in\mathcal D_k(\mS): A\subseteq Z
\bigr\}
\le
\frac{q^{\,n-k+2}-1}{q^2-1}.
\]

To derive the result, consider the set
\[
\mathcal I
:=
\bigl\{
(A,Z):
A\le \mS,\;
\dimq(A)=k-2,\;
Z\in\mathcal D_k(\mS),\;
A\subseteq Z
\bigr\}.
\]
Each $Z\in\mathcal D_k(\mS)$ contains exactly $\qbinom{k}{2}$
$(k-2)$-dimensional $\mathbb F_q$-subspaces. Therefore,
\[
|\mathcal I|
=
|\mathcal D_k(\mS)|
\genfrac{[}{]}{0pt}{}{k}{2}_q.
\]
On the other hand, since $\mS$ has $\mathbb F_q$-dimension $n$, the total number of
$(k-2)$-dimensional $\mathbb F_q$-subspaces of $\mS$ is $\qbinom{n}{k-2}$. Hence,
\[
|\mathcal{D}_k(\mS)|\qbinom{k}{2}=|\mathcal I|
\le
\genfrac{[}{]}{0pt}{}{n}{k-2}_q
\,
\frac{q^{\,n-k+2}-1}{q^2-1}
,
\]
and so
\[
|\mathcal D_k(\mS)|
\le
\frac{\genfrac{[}{]}{0pt}{}{n}{k-2}_q}{\genfrac{[}{]}{0pt}{}{k}{2}_q}\,
\frac{q^{\,n-k+2}-1}{q^2-1}
=\frac{\qbinom{n}{k-1}}{\qbinom{k}{1}}.
\]
This concludes the proof.
\end{proof}

Let $\cC$ be an $[n,k]_{q^m/q}$ near-MRD code and let
$\mS\in\Phi([\cC])$. Recall that $\mathcal D_k(\mS)$ is a set of $\F_q$-subspaces in $\mathcal{S}$ of dimension $k$ pairwise intersecting
  in at most a $(k-2)$-dimensional subspace.
Then, after identifying $\mS$ with $\F_q^n$ we may regard
$\mathcal D_k(\mS)$ as a subset of $\mathcal{G}_{k}(\F_q^n)$.
Hence, with respect to the subspace distance
\[
d_{\rm S}(X,Y)
=\dim_{\F_q}X+\dim_{\F_q}Y-2\dim_{\F_q}(X\cap Y),
\]
we obtain that $\mathcal D_k(\mS)$ has subspace distance at least 4 and it is an
$(n,|\mathcal{D}_k(\mathcal{S})|,4;k)_q$ constant-dimension code in the terminology of
subspace coding~\cite{koetterkschischang2008,etzionvardy2011}. In particular,
if $A_q(n,4;k)$ denotes the largest possible size of a constant-dimension
code in $\mathcal{G}_{k}(\F_q^n)$ with minimum subspace distance at least $4$, then
\[
|\mathcal D_k(\mS)|\leq A_q(n,4;k).
\]

The same property can be expressed in the language of $q$-packing designs.
A $P_q(t,k,n)$ \textit{packing design} is a family of $k$-dimensional subspaces of
$\F_q^n$ such that every $t$-dimensional subspace is contained in at most one
member of the family~\cite{braunreichelt2014}. Since two distinct elements of
$\mathcal D_k(\mS)$ cannot contain the same $(k-1)$-dimensional subspace,
$\mathcal D_k(\mS)$ is precisely a $P_q(k-1,k,n)$ packing design.

In this setting, Theorem~\ref{thm:bound-Dk-near-MRD} recover the so-called \textit{standard packing bound}.
Indeed, every member of $\mathcal D_k(\mS)$ contains exactly
$\qbinom{k}{k-1}=\qbinom{k}{1}$ subspaces of dimension $k-1$, and each such
$(k-1)$-subspace can occur in at most one member. 

This bound is also closely related to \textit{Delsarte's code--anticode bound} for the
Grassmann scheme see, e.g., \cite[Theorem~3]{khaleghi2009}. In the range $k\leq n/2$, the anticode bound for minimum
subspace distance $4$ is exactly
\[
A_q(n,4;k)
\leq
\frac{\qbinom{n}{k-1}}{\qbinom{k}{1}},
\]
so it coincides with bound in Theorem~\ref{thm:bound-Dk-near-MRD}. More generally, using the
duality $A_q(n,4;k)=A_q(n,4;n-k)$, the anticode bound can be written in the
symmetric form
\[
A_q(n,4;k)
\leq
\frac{\qbinom{n}{k}}
     {\qbinom{\max\{k,n-k\}+1}{1}}.
\]
Consequently, when $k>n/2$ the dual form of the anticode bound may already be
strictly stronger than the one-sided packing bound obtained by counting
$(k-1)$-subspaces. Furthermore, even when $k\leq n/2$, sharper bounds for
$A_q(n,4;k)$ are known for several parameter sets, and any such bound applies
a fortiori to $\mathcal D_k(\mS)$. For example, $A_2(6,4;3)=77$ by~\cite{honoldkiermaierkurz2015}, whereas the bound in Theorem~\ref{thm:bound-Dk-near-MRD} gives $\qbinom{6}{2}/\qbinom{3}{1}=93$. Hence, for these parameters one has the strictly stronger estimate
$|\mathcal D_3(\mS)|\leq 77$.

\medskip

The following two examples illustrate the upper bound in
Theorem~\ref{thm:bound-Dk-near-MRD}: the first gives a
family attaining the bound, whereas the second shows that the bound
need not be attained. Equality in the  bound of Theorem~\ref{thm:bound-Dk-near-MRD}  has a strong design-theoretic meaning. Indeed, $\mathcal{D}_k(\mathcal{S})$ attains the packing bound if and only if every $(k-1)$-dimensional subspace of $\mS$ is contained
in exactly one member of $\mathcal D_k(\mS)$. Thus, $\mathcal D_k(\mS)$ is a $q$-Steiner system
$S_q(k-1,k,n)$; see, e.g., \cite{braun2016, etzionvardy2011}.
Conversely, every $S_q(k-1,k,n)$ is a constant-dimension code with minimum
subspace distance at least $4$ and attains the packing bound.

\begin{example}\label{exam:1}
Let $2<n<m$
with $n\mid m$, so that $\F_{q^n}\subseteq\F_{q^m}$, and choose
$s$ such that $\gcd(s,n)=2$. Consider
$$\mS=\{(x,x^{q^s}):x\in\F_{q^n}\}\subseteq\F_{q^m}^2$$
and let $\cC\in\Psi([\mS])$.
Then $\dim_{\F_q}(\mS)=n$. Every $1$-dimensional vector subspace of $\F_{q^m}^2$ intersecting the $\F_q$-vector space $\mS$ nontrivially is of the form $H_x=\langle(x,x^{q^s})\rangle_{\F_{q^m}}$, for some
$x\in\F_{q^n}^*$, and $H_x\cap\mS=\{\lambda(x,x^{q^s}):\lambda\in\F_{q^2}\}$. Hence the maximum weight of a $1$-dimensional vector subspace with respect to $\mS$ is $2$, and therefore
$d(\cC)=n-2$.

Moreover, $\mS$ spans $\F_{q^m}^2$ over $\F_{q^m}$. Indeed, if
$a\in\F_{q^n}\setminus\F_{q^2}$, then $a^{q^s}\neq a$, and hence
$(1,1)$ and $(a,a^{q^s})$ are $\F_{q^m}$-linearly independent.
Thus $d_2(\cC)=n$, and $\cC$ is an $[n,2,n-2]_{q^m/q}$ near-MRD code.

Since $H_x=H_y$ if and only if
$\frac{x}{y}\in\F_{q^2}^*$, we have that
\[
|\mathcal{D}_2(\mS)|
=\frac{q^n-1}{q^2-1}
=\frac{\qbinom{n}{1}}{\qbinom{2}{1}},
\]
and so the upper bound in Theorem is attained. Note that the condition $\dim_{\F_q}(X\cap Y)\leq k-2$ shows that
$\mathcal D_2(\mS)$ is a partial $2$-spread of $\F_q^n$.
\end{example}

\begin{example}
Let $\alpha\in\F_{16}\setminus\F_2$ and consider the $[3,2]_{2^4/2}$ code $\cC$
generated by
\[
\mathbf G=
\begin{pmatrix}
1 & \alpha & 0\\
0 & 0 & 1
\end{pmatrix},
\]
and let $\mS\in\Phi([\cC])$.
Then
$\cC=\{(a,a\alpha,b):a,b\in\F_{16}\}$.
Taking $a=0$ and $b\neq0$, the codeword $(0,0,b)$ has rank $1$ and, hence,
$d_1(\cC)=1$.

Moreover, $\mathbf G$ has full $\F_q$-column rank, so $\cC$ is
nondegenerate. Hence, $d_2(\cC)=n=3$.  Therefore
$\cC$ is a
$[3,2,1]_{2^4/2}$ near-MRD code. Now, let $H_0=\langle(1,0)\rangle_{\F_{16}}$. Then
$\mS\cap H_0=\langle(1,0),(\alpha,0)\rangle_{\F_2}$ has
$\F_2$-dimension $2$. For any other 1-dimensional
$\F_{16}$-subspace $H\neq H_0$, we have
$\dim_{\F_2}(\mS\cap H)\leq 1$. Indeed, the four nonzero vectors of
$\mS\setminus H_0$ are
$(0,1),(1,1),(\alpha,1)$ and $(1+\alpha,1)$, and no two of them belong
to the same 1-dimensional $\F_{16}$-subspace. Hence $H_0$ is the
unique 1-dimensional $\F_{16}$-subspace $H$ satisfying
$\dim_{\F_2}(\mS\cap H)=2$. Therefore,
$\mathcal D_2(\mS)=\{\mathcal{S} \cap H_0\}$ and $|\mathcal D_2(\mS)|=1$.
On the other hand, Theorem~\ref{thm:bound-Dk-near-MRD} gives
$|\mathcal D_2(\mS)|\leq \left \lfloor 7/3 \right \rfloor =2$.
Hence the upper bound is not attained.

\end{example}

\section{Determining $|\mathcal{D}_k(\mathcal{S})|$ for optimal near-MRD codes}\label{determining}

In this section, we determine $|\mathcal D_k(\mS)|$ for $q$-systems associated with a near-MRD codes in the case where $n > m$, and hence for optimal ones. Our approach is based on the properties of near-MRD codes and their adjoint codes. For this reason, we introduce here the necessary notation on these latetr that will be used throughout this section.

Let $\cC\subseteq\F_q^{n\times m}$ be a rank-metric code. The \textit{adjoint code} of $\cC$ is defined by
\[
\cC^\top:=\{M^t:M\in\cC\}\subseteq\F_q^{m\times n},
\]
where $M^t$ denotes the transpose of $M$. For $0\leq i\leq \min\{n,m\}$, let $A_i(\cC)$ denote the number of codewords of rank $i$ in $\cC$. Since $\rk(M)=\rk(M^t)$ for every $M\in\cC$, we have
$
A_i(\cC)=A_i(\cC^\top)
$
for every $i$, and consequently
$
d(\cC)=d(\cC^\top).
$

Recall that the Delsarte dual of $\cC$ is
\[
\cC^\perp:=\{M\in\F_q^{n\times m}:\langle M,N\rangle=0,\ \forall\,N\in\cC\},
\]
where $\langle M,N\rangle=\Tr(MN^t)$. Moreover, the adjoint and the Delsarte dual satisfy $(\cC^\top)^\perp=(\cC^\perp)^\top$.

Let $\cC$ be an $[n,k,d]_{q^m/q}$ rank-metric code and let $\Gamma$ be an ordered $\F_q$-basis of $\F_{q^m}$. With the same notation introduced in Section~\ref{subsec:rm}, $\cC_\Gamma:=\{\Gamma(\vv):\vv\in\cC\}\subseteq\F_q^{n\times m}$ is an $[n\times m,mk,d]_q$ rank-metric code. The matrix representation preserves the rank weight of each codeword, so that $A_i(\cC)=A_i(\cC_\Gamma)$ for every $i$. Moreover, since $\rk(M)=\rk(M^t)$ for every $M\in\cC_\Gamma$, we have $A_i(\cC_\Gamma)=A_i(\cC_\Gamma^\top)$. Therefore, for every $i$, we have
\[
A_i(\cC)=A_i(\cC_\Gamma^\top).
\]

\begin{lemma}\label{re:adjoint}
Let $\cC$ be an $[n,k]_{q^m/q}$ near-MRD code and $\Gamma$ be an ordered basis of $\F_{q^m}$ over $\F_q$. Denote by $\Gamma^\ast$ the trace-dual basis of $\Gamma$. Then
\[
d(\cC_\Gamma^\top)+d((\cC_{\Gamma}^\top)^\perp)=n.
\]

\end{lemma}
\begin{proof}
By~\cite[Theorem~21]{ravagnani}, we have that $\cC_\Gamma^\perp=(\cC^\perp)_{\Gamma^\ast}.$
Since the matrix representation preserves rank weights, it follows that
\[
d(\cC_\Gamma)=d(\cC)
\qquad\text{and}\qquad
d(\cC_\Gamma^\perp)
=d((\cC^\perp)_{\Gamma^\ast})
=d(\cC^\perp).
\]
As $\cC$ is near-MRD, $d(\cC)+d(\cC^\perp)=n$ and hence $d(\cC_\Gamma)+d(\cC_\Gamma^\perp)=n.$

Finally, matrix transposition preserves rank and
$(\cC_\Gamma^\top)^\perp
=(\cC_\Gamma^\perp)^\top.$
Therefore
$d(\cC_\Gamma^\top)
=d(\cC_\Gamma)$
and
$
d((\cC_\Gamma^\top)^\perp)
=d(\cC_\Gamma^\perp),$
so that
$
d(\cC_\Gamma^\top)
+d((\cC_\Gamma^\top)^\perp)=n.
$
\end{proof}

The following lemma establishes the relation between the number of minimum-rank codewords of a near-MRD code and $|\mathcal D_k(\mS)|$.

\begin{proposition}
\label{reAsize}
Let $\cC$ be a $[n,k]_{q^m/q}$ near-MRD code and let
$\mS\in\Phi([\cC])$. Then
\[
A_{n-k}(\cC)=(q^m-1)|\mathcal D_k(\mS)|.
\]
\end{proposition}

\begin{proof}
Every nonzero codeword $\mathbf c\in\cC$ determines a hyperplane
$H_{\mathbf c}$ of $\F_{q^m}^{k}$ in the correspondence between
linear codes and $q$-systems. More precisely, if $\mathbf{c}= \mathbf{x} \mathbf{G}$, $\mathbf{x} \in \F_{q^m}^l \setminus \{\mathbf{0}\}$, then $H_{\mathbf{c}}= \mathbf{x}^\perp$.   Moreover, $H_\mathbf{c}=H_{\lambda \mathbf{c}}$ for any $\lambda \in \F_{q^m}^*$ and
\[
\rk(\mathbf c)
=n-\dim_{\F_q}(H_{\mathbf c}\cap\mS).
\]
Therefore,
\[
\rk(\mathbf c)=n-k
\iff
\dim_{\F_q}(H_{\mathbf c}\cap\mS)=k.
\]

Since every hyperplane of $\F_{q^m}^{k}$ is represented by exactly
$q^m-1$ nonzero scalar multiples and by Lemma~\ref{lem:hyp-cyc}, we have
\[
A_{n-k}(\cC)
=(q^m-1)
\cdot \#\{H: H\leq \F_{q^m}^k \text{~and~}\dim_{\F_q}(H\cap\mS)=k\}
=(q^m-1)|\mathcal D_k(\mS)|.
\]
\end{proof}

\begin{theorem}\label{th:sizebound}Let $n> m$. 
     Let $\cC$ be an $[n,k,d]_{q^m/q}$ near-MRD code and let $\mS \in \Phi([\cC])$. Then,
\[
|\mathcal D_k(\mS)|
=\frac{q^{(n-m)d}-1}{q^m-1}
\genfrac{[}{]}{0pt}{}{m}{d}_q
.
\]
\end{theorem}
\begin{proof}
By \cite[Corollary 28]{dela18} and Proposition \ref{re:adjoint}, we have that \[
A_{d}(\cC)
=
\left(q^{(n-m)d}-1\right)
\genfrac{[}{]}{0pt}{}{m}{d}_q.
\]
Hence, by Proposition \ref{reAsize}, the result follows.
\end{proof}

By \cite[Theorem~27]{dela18}, the rank distribution of a matrix code is determined by its parameters together with the coefficients
\[
A_d(\cC),\ A_{d+1}(\cC),\ldots,A_{n-d^\perp}(\cC),
\]
where $d^\perp$ denotes the minimum distance of its Delsarte dual. In this case $d^\perp=k$, and hence $n-d^\perp=n-k=d$. Thus this list consists of the single coefficient $A_d(\mathcal C)$. Consequently, for $n\le m$, the entire rank distribution of a near-MRD code is determined by $n,m,k,q$ and the number $A_d(\mathcal C)$ of its minimum-rank codewords. Hence, as a consequence pf Proposition \ref{reAsize}, the rank distribution of $\cC$ is completely determined whenever $|\mathcal D_k(\mS)|$ is known as it happens for the construction in Example~\ref{exam:1}.

\section{Conclusions and open questions}\label{sec:conclusions}

In this paper, we studied the $q$-matroids associated with near-MRD codes and introduced the class of near-uniform $q$-matroids. We determined their rank function and characterized their cyclic flats providing an upper bound on their number. When $n > m$, and hence in particular for optimal near-MRD codes, we determined the exact number of nontrivial cyclic flats of the associated near-uniform $q$-matroids.

Several questions remain open. A first one concerns the existence and construction of optimal near-MRD codes. By Proposition~\ref{upperbound-near}, the exceptional case $m=2k-2$ is equivalent to the existence of maximum $(k-2)$-scattered $\mathbb{F}_q$-subspaces of $\F_{q^{2k-2}}^k$. For $k=3$, the existence of maximum scattered subspaces of $\F_{q^4}^3$ has already been established in~\cite{cmpz2017}, which gives optimal $[6,3,3]_{q^4/q}$ near-MRD codes. The case $k=4$ is covered by the construction of maximum $2$-scattered subspaces of $\F_{q^6}^4$ for $q$ an odd power of two in~\cite{bggm2025}, which yields optimal $[8,4,4]_{q^6/q}$ near-MRD codes. To the best of our knowledge, the existence problem for maximum $(k-2)$-scattered subspaces of $\F_{q^{2k-2}}^k$ remains unsolved for $k\geq 5$, suggesting the following problem.\\
Finally, Theorem~\ref{thm:bound-Dk-near-MRD} provides an upper bound on $|\mathcal D_k(\mS)|$ for general parameters. 
Theorem~\ref{th:sizebound} determines $|\mathcal{D}_k(\mS)|$ for $n>m$. For $n\le m$, the upper bound in Theorem~\ref{thm:bound-Dk-near-MRD} can be attained for certain parameters, but its sharpness is not determined in general. 
 The relation between $|\mathcal D_k(\mS)|$ and the number of minimum-rank codewords, it also yields an upper bound on $A_d(\cC)$. Hence, determining sharper bounds for $|\mathcal D_k(\mS)|$ for general parameters $n,m,k$ and $q$, and consequently for $A_d(\cC)$, remains an open problem.

\section{Acknowledgment}
This research has been partially supported by the Italian National Group for Algebraic and Geometric Structures and their Applications (GNSAGA -- INdAM).

\end{document}